\documentclass{amsart}

\usepackage{graphicx,cite,mathrsfs}
\usepackage{amscd,amsfonts,amsmath,amssymb,amsthm}
\usepackage{hyperref}

\begin{document}

\newtheorem{theorem}{Theorem}[section]
\newtheorem{corollary}[theorem]{Corollary}
\newtheorem{definition}[theorem]{Definition}
\newtheorem{lemma}[theorem]{Lemma}
\newtheorem{proposition}[theorem]{Proposition}
\newtheorem{example}[theorem]{Example}
\theoremstyle{remark}
\newtheorem{remark}{Remark}[section]

\title[New orthogonality relations for super-Jack polynomials]{New orthogonality relations for super-Jack polynomials and an associated Lassalle--Nekrasov correspondence}

\author{Martin Halln\"as}
\email{hallnas@chalmers.se}
\address{Department of Mathematical Sciences, Chalmers University of Technology and the University of Gothenburg, SE-412 96 Gothenburg, Sweden}

\thanks{Supported by the Swedish Research Council (Project-id 2018-04291).}

\date{\today}


\begin{abstract}
The super-Jack polynomials, introduced by Kerov, Okounkov and Olshanski, are polynomials in $n+m$ variables, which reduce to the Jack polynomials when $n=0$ or $m=0$ and provide joint eigenfunctions of the quantum integrals of the deformed trigonometric Calogero--Moser--Sutherland system. We prove that the super-Jack polynomials are orthogonal with respect to a bilinear form of the form $(p,q)\mapsto (L_pq)(0)$, with $L_p$ quantum integrals of the deformed rational Calogero--Moser--Sutherland system. In addition, we provide a new proof of the Lassalle--Nekrasov correspondence between deformed trigonometric and rational harmonic Calogero--Moser--Sutherland systems and infer orthogonality of super-Hermite polynomials, which provide joint eigenfunctions of the latter system.
\end{abstract}

\maketitle

\tableofcontents

\section{Introduction}
As is well-known, the (symmetric) Jack polynomials $P^{(\theta)}_\lambda(x_1,\ldots,x_N)$, labelled by a partition $\lambda$ and depending rationally on a parameter $\theta$, have numerous remarkable properties. In particular, they form an orthogonal system with respect to the weight function
\begin{equation}
\label{DeltaN}
\Delta_N(x;\theta) = \prod_{1\leq i<j\leq N}[(1-x_i/x_j)(1-x_j/x_i)]^\theta
\end{equation}
on the $N$-torus; and they are joint eigenfunctions of the trigonometric Calogero--Moser--Sutherland operator
\begin{equation}
\label{cLN}
\mathcal{L}_N(x;\theta) = \sum_{i=1}^N \left(x_i\frac{\partial}{\partial x_i}\right)^2 + \theta\sum_{1\leq i<j\leq N}\frac{x_i+x_j}{x_i-x_j}\left(x_i\frac{\partial}{\partial x_i}-x_j\frac{\partial}{\partial x_j}\right)
\end{equation}
and its quantum integrals. For historical accounts of Henry Jack, his construction of the polynomials that now bear his name as well as many more of their properties; see e.g.~the proceedings \cite{JHLM06} and Macdonald's book \cite{Mac95}. We note that the parameter $\theta$ is the inverse of the parameter $\alpha$ used, in particular, in \cite{Mac95}.

In addition to the root system generalisations introduced by Olshanetsky and Perelomov \cite{OP83}, the operator \eqref{cLN} allows for non-symmetric integrable generalisations, such as
\begin{equation}
\label{cLnm}
\mathcal{L}_{n,m}(x,y;\theta) = \mathcal{L}_n(x;\theta) - \theta\mathcal{L}_m(y;1/\theta) - \sum_{i=1}^n\sum_{j=1}^m \frac{x_i+y_j}{x_i-y_j}\left(x_i\frac{\partial}{\partial x_i}+\theta y_j\frac{\partial}{\partial y_j}\right).
\end{equation}
This operator was first introduced in the $m=1$ case by Chalykh, Feigin and Veselov \cite{CFV98}, who proved integrability, expressed the operator in terms of a particular deformation of the root system $A_n$ and introduced the terminology {\it deformed Calogero--Moser--Sutherland operator}. Shortly thereafter, Sergeev \cite{Ser01,Ser02} wrote down the operator \eqref{cLnm} for general $m$ and showed that it has the so-called super-Jack polynomials $SP^{(\theta)}_\lambda((x_1,\ldots,x_n),(y_1,\ldots,y_m))$, introduced a few years earlier by Kerov, Okounkov and Olshanski \cite{KOO98}, as eigenfunctions. These results on integrability and eigenfunctions were then extended to arbitrary $m>1$ and all quantum integrals, respectively, by Sergeev and Veselov \cite{SV04,SV05}.

At an early stage, it became clear that the orthogonality of the Jack polynomials with respect to the weight function \eqref{DeltaN} could not be directly generalised to the super-Jack polynomials. Indeed, starting from Sergeev's work and proceeding formally, one is naturally led to the weight function
$$
\Delta_{n,m}(x,y;\theta) = \frac{\Delta_n(x;\theta)\Delta_m(y;1/\theta)}{\prod_{i=1}^n\prod_{j=1}^m(1-x_i/y_j)(1-y_j/x_i)},
$$
which is manifestly singular along the hyperplane $x_i=y_j$ for all $1\leq i\leq n$ and $1\leq j\leq m$ and any value of $\theta$.

Together with Atai and Langmann \cite{AHL19}, we circumvented this problem by integrating $x$ and $y$ over tori with different radii. Even though this meant dealing with a complex-valued weight function, we could prove that the resulting sesquilinear form is Hermitian on the space spanned by the super-Jack polynomials. Moreover, we identified its kernel as the subspace spanned by the $SP^{(\theta)}_\lambda$ with $(m^n)\not\subset\lambda$, and showed that the form descends to a positive definite inner product on the corresponding factor space. As a consequence, we obtained a Hilbert space interpretation of the deformed Calogero--Moser--Sutherland operator \eqref{cLnm}. These results were essentially all inferred from orthogonality relations, including an explicit formula for (quadratic) norms, for super-Jack polynomials.

In this paper, we prove orthogonality of the super-Jack polynomials with respect to another rather different bilinear form. Initially, we will define it in terms of deformed rational Calogero--Moser--Sutherland operators, but, for suitable values of $\theta$, it also has the integral representation
\begin{multline}
\label{intRep}
(p,q)_{n,m}\\
= M_{n,m}^{-1} \int_{\mathbb{R}^n+i\xi}\int_{\mathbb{R}^m+i\eta} \big(e^{-L_{n,m}/2}p\big)(x,y)\big(e^{-L_{n,m}/2}q\big)(x,y)\frac{e^{-x^2/2+\theta^{-1}y^2/2}}{A_{n,m}(x,y)}dxdy,
\end{multline}
with
$$
A_{n,m}(x,y) = \prod_{1\leq i<j\leq n}(x_i-x_j)^{-2\theta}\cdot \prod_{1\leq i<j\leq m}(y_i-y_j)^{-2/\theta}\cdot \prod_{i=1}^n\prod_{j=1}^m(x_i-y_j)^2,
$$
and where $M_{n,m}$ is a normalisation constant, $\xi\in\mathbb{R}^n$ and $\eta\in\mathbb{R}^m$ are chosen such that all singularities are avoided, $x^2=x_1^2+\cdots+x_n^2$, $y^2=y_1^2+\cdots+y_m^2$ and $L_{n,m}$ denotes the rational limit of \eqref{cLnm}.

When $m=0$ this form amounts to a restriction of the $A_{n-1}$-instance of Dunkl's \cite{Dun91} bilinear form $[p,q]_\theta=(p(D)q)(0)$ to symmetric polynomials $p,q$. Here, $p(D)$ denotes the operators obtained from $p(x)$ by substituting for $x_i$ a Dunkl operator $D_{i,n}$ (see Eq.~\eqref{DiN} below). An integral representation similar to \eqref{intRep} was established in the same paper by Dunkl. Corresponding orthogonality results for (non-symmetric) Jack polynomials were later deduced by Baker and Forrester \cite{BF98} as well as R\"osler \cite{Ros98}.

Restricting attention further to $\theta=0$, we recover the bilinear form $[p,q]_\partial:=(p(\partial)q)(0)$, where $p(\partial)=p(\partial/\partial x_1,\ldots,\partial/\partial x_n)$, for which Macdonald \cite{Mac80} proved $[p,q]_\partial=(2\pi)^{-n/2}\int_{\mathbb{R}^n}(e^{-\Delta/2}p)(x)(e^{-\Delta/2}q)(x)e^{-x^2/2}dx$. In this special case, Jack polynomials are simply symmetric monomials $m_\lambda$ and their orthogonality on the $n$-torus $T^n$ simply amounts to $\int_{T^n}m_{\mu}(x)m_\lambda(x^{-1})dx=|S_n(\mu)|\cdot\delta_{\mu\lambda}$, with $\delta_{\mu\lambda}$ the Kronecker delta, $S_n(\mu)$ the orbit of $\mu$ under the standard action of $S_n$ and $dx$ the normalised Haar measure on $T^n$.

Moreover, if $m=1$ and $\theta\in-\mathbb{N}$, then $(\cdot,\cdot)_{n,m}$ amounts to a (restricted) instance of the bilinear form on quasi-invariants introduced and studied by Feigin and Veselov \cite{FV02,FV03}, with an integral representation of the form \eqref{intRep} obtained in \cite{FHV13}.

From \eqref{intRep} and our orthogonality results for the super-Jack polynomials, we infer that the polynomials
\begin{equation}
\label{SH}
SH^{(\theta)}_\lambda(x,y) := e^{-L_{n,m}/2}SP^{(\theta)}_\lambda(x,y)
\end{equation}
form an orthogonal system on $(\mathbb{R}^n+i\xi)\times(\mathbb{R}^m+i\eta)$ with respect to the weight function $e^{-x^2/2+\theta^{-1}y^2/2}/A_{n,m}(x,y)$. These polynomials coincide with the so-called super-Hermite polynomials introduced in \cite{DH12} (see Prop.~6.7); and for $m=0$ they amount to the generalised Hermite polynomials first introduced by Lassalle \cite{Las91} and studied in further detail by Baker and Forrester \cite{BF97} and van Diejen \cite{vDie97}.

In recent joint work with Feigin and Veselov \cite{FHV21}, we showed that $e^{-L_{n,m}/2}$ intertwines quantum integrals of deformed Calogero--Moser--Sutherland operators of trigonometric and rational harmonic type, and so \eqref{SH} can be interpreted as a correspondence between joint eigenfunctions of these two integrable systems. Motivated by Lassalle's construction of generalised Hermite polynomials and Nekrasov's \cite{Nek97} discovery that the ordinary A type trigonometric and rational harmonic Calogero--Moser--Sutherland systems are essentially equivalent, we proposed for such a correspondence the terminology {\it Lassalle--Nekrasov correspondence}.

Our results entail that this particular example of a Lassalle--Nekrasov correspondence between deformed Calogero--Moser--Sutherland systems is isometric in the sense that the operator $e^{-L_{n,m}/2}$ becomes an isometry when its domain is equipped with the bilinear form $(\cdot,\cdot)_{n,m}$ and its codomain with the form given by the right-hand side of \eqref{intRep} with $e^{-L_{n,m}/2}$ removed. In the undeformed case $m=0$ this can be seen already in \cite{Dun91}; see also \cite{Ros98}.

We conclude this introduction with an outline of the remainder of the paper. In Section \ref{Sec:Prel}, we review definitions and results pertaining to (super-)Jack polynomials and (deformed) rational Calogero--Moser--Sutherland operators that we make use of. Readers familiar with these matters may wish to skip ahead to Section \ref{Sec:Hyperg}, where particular instances of generalised hypergeometric series associated with Jack polynomials, Jack symmetric functions or super-Jack polynomials are recalled and shown to be joint eigenfunctions of (deformed) rational Calogero--Moser--Sutherland quantum integrals. In Section \ref{Sec:Form}, we introduce the relevant bilinear form and establish a number of its basic properties, identify its reproducing kernel as a generalised hypergeometric series and make the integral representation \eqref{intRep} precise. We infer orthogonality relations for super-Jack polynomials in Section \ref{Sec:Orth} and our results on the Lassalle--Nekrasov correspondence between deformed trigonometric and rational harmonic Calogero--Moser--Sutherland systems are contained in Section \ref{Sec:LN}. In the final Section \ref{Sec:Out}, we provide a brief outlook on possible directions for future research; and in Appendix \ref{App:Conv}, we study convergence properties of the pertinent generalised hypergeometric series associated with super-Jack polynomials.

\subsection*{Notation}
We use the convention $\mathbb{N}=\{1,2,\ldots\}$ and write $\mathbb{Z}_{\geq 0}$ for $\mathbb{N}\cup\{0\}$. To a large extent, we follow Macdonald's book \cite{Mac95} for notation and terminology from the theory of symmetric functions. One notable exception is our use of the parameter $\theta$, which is the inverse of the parameter $\alpha$.

\section{Preliminaries}
\label{Sec:Prel}
In this section, we specify our basic notations and review terminology and results relating to (super-)Jack polynomials and (deformed) rational Calogero--Moser--Sutherland operators that we rely on in our construction of a bilinear form in Section \ref{Sec:Form}.

\subsection{Symmetric functions}
Throughout the paper, we shall work over the complex numbers $\mathbb{C}$. Therefore, we write simply $\Lambda_N$ for the algebra of symmetric polynomials in $N$ independent variables with complex coefficients and $\Lambda_N^k$ for the subspace consisting of homogeneous symmetric polynomials of degree $k$.

Given an infinite sequence $x=(x_1,x_2,\ldots)$ of independent variables, we recall that a homogeneous symmetric function of degree $k$ with complex coefficients can be viewed as a formal power series
$$
f(x) = \sum_\alpha c_\alpha x^\alpha,
$$
where the sum extends over all sequences $\alpha=(\alpha_1,\alpha_2,\ldots)$ of non-negative integers such that $\alpha_1+\alpha_2+\cdots=k$, each $c_\alpha\in\mathbb{C}$, $x^\alpha=x_1^{\alpha_1}x_2^{\alpha_2}\cdots$ and $f(x_{\sigma(1)},x_{\sigma(2)},\ldots)=f(x_1,x_2,\ldots)$ for any permutation $\sigma$ of $\mathbb{N}$; see e.g.~Chapter 7 in Stanley's book \cite{Sta99}.

We let $\Lambda^k$ denote the space of all (complex) homogenous symmetric functions of degree $k$. If $f\in\Lambda^k$ and $g\in\Lambda^l$, it is clear that $fg\in\Lambda^{k+l}$, which implies that
$$
\Lambda := \Lambda^0\oplus \Lambda^1\oplus\cdots
$$
has a natural structure of a graded $\mathbb{C}$-algebra, called the ($\mathbb{C}$-)algebra of symmetric functions.

Partitions $\lambda=(\lambda_1,\lambda_2,\ldots)$ of $k$, i.e.~sequences of non-negative integers $\lambda_1,\lambda_2,\ldots$ such that $\lambda_1+\lambda_2+\cdots=k$, naturally label bases in $\Lambda^k$. There are numerous important examples and we use, in particular, the following:

\begin{enumerate}
\item Monomial symmetric functions:
$$
m_\lambda = \sum_\alpha x^\alpha,
$$
with the sum extending over all distinct permutations $\alpha=(\alpha_1,\alpha_2,\ldots)$ of the parts of $\lambda=(\lambda_1,\lambda_2,\ldots)$.

\item Power sum symmetric functions:
$$
p_\lambda = p_{\lambda_1}p_{\lambda_2}\cdots
$$
with $p_0\equiv 1$ and
$$
p_r = x_1^r+x_2^r+\cdots\ \ \ (r\in\mathbb{N}).
$$
\end{enumerate}

\subsection{Jack symmetric functions}
We recall that, as $\lambda$ runs through all partitions $\lambda$ of weight $k\in\mathbb{Z}_{\geq 0}$, the (monic) Jack symmetric functions $P^{(\theta)}_\lambda$ form an orthogonal basis in $\Lambda^k$ with respect to the the scalar product given by
\begin{equation}
\label{pOrth}
\langle p_\lambda,p_\mu\rangle = \theta^{-l(\lambda)}z_\lambda\delta_{\lambda\mu},
\end{equation}
where $l(\lambda)$ denotes the length of $\lambda$ and $z_\lambda=\prod_{i\geq 1}i^{m_i}\cdot m_i!$ with $m_i=m_i(\lambda)$ the multiplicity of $i$ in $\lambda$.

More specifically, they are characterized by the following two properties:
\begin{enumerate}
\item $P^{(\theta)}_\lambda(x)=m_\lambda(x)+\sum_{\mu<\lambda}u^{(\theta)}_{\lambda\mu}m_\mu(x)$,

\item $\big\langle P^{(\theta)}_\lambda,P^{(\theta)}_\mu\big\rangle=0$ whenever $\lambda\neq\mu$,
\end{enumerate}
where $<$ refers to the dominance order. (Note that, since the dominance order is only a partial order, it is far from obvious that such symmetric functions actually exist.) The coefficients $u^{(\theta)}_{\lambda\mu}$ are known to be rational functions of $\theta$, with poles occurring only at negative rational values of $\theta$. To ensure that such poles are not encountered and also that we avoid the singularity at $\theta=0$ in \eqref{pOrth}, we consider throughout the paper only values of $\theta$ not of the form
\begin{equation}
\label{thetaValsPrel}
\theta = i/j,\ \ \ i=-\mathbb{Z}_{\geq 0},\ \ j\in\mathbb{N}.
\end{equation}

The Jack symmetric functions were first introduced by Jack \cite{Jac70}, while the above characterisation is due to Macdonald \cite{Mac95}.

When setting $x_i=0$ for $i>N\in\mathbb{N}$, the Jack symmetric function $P^{(\theta)}_\lambda(x)$ specialises to the Jack polynomial $P^{(\theta)}_\lambda(x_1,\ldots,x_N)$ in the remaining $N$ variables $x_1,\ldots,x_N$.

For later reference, we record Stanley's \cite{Sta89} (see also \cite{Mac95}) quadratic norms formula
\begin{equation}
\label{qNorms}
\big\langle P^{(\theta)}_\lambda,P^{(\theta)}_\lambda\big\rangle = \frac{1}{b_\lambda^{(\theta)}},\ \ \ b_\lambda^{(\theta)} = \prod_{s\in\lambda}\frac{a(s)+\theta l(s)+\theta}{a(s)+1+\theta l(s)},
\end{equation}
and specialisation formula
\begin{equation}
\label{spec}
\epsilon_X\left(P_\lambda^{(\theta)}\right) = \prod_{s\in\lambda}\frac{\theta X+a^\prime(s)-\theta l^\prime(s)}{a(s)+\theta l(s)+\theta},
\end{equation}
where $\epsilon_X:\Lambda\to\mathbb{C}[X]$ denotes the homomorphism given by
\begin{equation}
\label{epsX}
\epsilon_X:p_r\mapsto X\ (r\in\mathbb{N}).
\end{equation}
In particular, the value at $x=1^N$ is obtained by setting $X=N$.

\subsection{Super-Jack polynomials}
For $n,m\in\mathbb{Z}_{\geq 0}$, we let
$$
P_{n,m} = \mathbb{C}[x_1,\ldots,x_n,y_1,\ldots,y_m].
$$
From \cite{SV04,SV05}, we recall its subalgebra $\Lambda_{n,m}$, consisting of all polynomials $p(x,y)$ that are symmetric in $x=(x_1,\ldots,x_n)$, symmetric in $y=(y_1,\ldots,y_m)$ and satisfy the quasi-invariance condition
\begin{equation}
\label{qInv}
\left(\frac{\partial}{\partial x_i}+\theta\frac{\partial}{\partial y_j}\right)p(x,y) = 0
\end{equation}
on each hyperplane $x_i=y_j$ with $i=1,\ldots,n$ and $j=1,\ldots,m$; the (for generic values of $\theta$) surjective homomorphism
\begin{equation}
\label{varphinm}
\varphi_{n,m}: \Lambda\to\Lambda_{n,m},\ p_r\mapsto p_{r,\theta}(x,y)\ (r\in\mathbb{N}),
\end{equation}
given in terms of the deformed power (or Newton) sums
$$
p_{r,\theta}(x,y) = \sum_{i=1}^n x_i^r - \frac{1}{\theta}\sum_{j=1}^m y_j^r;
$$
and the fact that the kernel of $\varphi_{n,m}$ is spanned by the Jack symmetric functions $P^{(\theta)}_\lambda$ labelled by the partitions $\lambda\notin H_{n,m}$, where
$$
H_{n,m} = \{\lambda\in\mathscr{P}\mid \lambda_{n+1}\leq m\},
$$
(with $\mathscr{P}$ denoting the set of all partitions).

For $\lambda\in H_{n,m}$, the super-Jack polynomial $SP^{(\theta)}_\lambda(x,y)\in\Lambda_{n,m}$ is given by
\begin{equation}
\label{SP}
SP^{(\theta)}_\lambda(x,y) = \varphi_{n,m}\big(P^{(\theta)}_\lambda\big).
\end{equation}
We note that they were originally introduced by Kerov, Okounkov \and Olshanski \cite{KOO98} in the setting of infinitely many variables and further studied by Sergeev and Veselov \cite{SV04,SV05} in the present context of $n+m$ variables.

\subsection{Rational Calogero--Moser--Sutherland operators}
It is readily seen that when substituting $x_i\to e^{i2\pi x_i/\mu}$, rescaling by $-(2\pi/\mu)^2$ and taking the period $\mu\to\infty$ the trigonometric Calogero--Moser--Sutherland operator \eqref{cLN} degenerates to its rational counterpart
\begin{equation}
\label{LN}
L_N := \sum_{i=1}^N\frac{\partial^2}{\partial x_i^2}+2\theta\sum_{1\leq i<j\leq N}\frac{1}{x_i-x_j}\left(\frac{\partial}{\partial x_i}-\frac{\partial}{\partial x_j}\right).
\end{equation}

We rely on the observation that the algebra of symmetric quantum integrals of this operator is given by the differential operators
\begin{equation}
\label{Lp}
L_{p,N} := \mathrm{Res}\big(p\big(D_{1,N},\ldots,D_{N,N}\big)\big)\ \ \ (p\in\Lambda_N),
\end{equation}
where $\mathrm{Res}$ stands for restriction to $\Lambda_N$ and $D_{i,N}$ ($i=1,\ldots,N$) are Dunkl differential-difference operators \cite{Dun89}, given by
\begin{equation}
\label{DiN}
D_{i,N} = \frac{\partial}{\partial x_i}+\theta\sum_{j\neq i}\frac{1}{x_i-x_j}(1-\sigma_{ij})
\end{equation}
with the transposition $\sigma_{ij}$ acting on a function $f(x)$ in $x=(x_1,\ldots,x_N)$ by exchanging variables $x_i$ and $x_j$. Letting
$$
L_N^{(r)} = L_{p_r(x_1,\ldots,x_N)}\ \ (r\in\mathbb{N}),
$$
where $L_{N}^{(2)}=L_{x_1^2+\cdots+x_N^2}=L_N$, we note that the algebra of symmetric quantum integrals is (freely) generated by $L_N^{(1)},\ldots,L_N^{(N)}$.

The above important observation is due to Heckman \cite{Hec91}, who worked in the more general setting of an arbitrary root system, whereas \eqref{LN}--\eqref{DiN} are associated with $A_{N-1}$.

For further details on (rational) Calogero--Moser--Sutherland systems and their integrability, see e.g.~\cite{Eti07,OP83,Rui99}.

\subsection{Dunkl operator at infinity}
We proceed to recall Sergeev and Veselov's \cite{SV15} notion of a Dunkl operator in infinitely many variables as well as corresponding expressions for rational Calogero--Moser--Sutherland operators at infinity.

Introducing a generator $p_0$, replacing the dimension $N$, the algebra $\bar{\Lambda}:=\Lambda[p_0]$ and homomorphisms
$$
\varphi_N: \bar{\Lambda}\to \Lambda_N,\, p_r\mapsto x_1^r+\cdots+x_N^r\ (r\in\mathbb{Z}_{\geq 0}, N\in\mathbb{N})
$$
are considered. Using a further generator $x$, the infinite-dimensional Dunkl operator $D_\infty:\bar{\Lambda}[x]\to\bar{\Lambda}[x]$ is given by
$$
D_\infty = \partial+\theta\Delta,
$$
with the derivation $\partial$ in $\bar{\Lambda}[x]$ characterised by Leibniz rule and
$$
\partial(x) = 1,\ \ \partial(p_r) = rx^{r-1}\ \ (r\in \mathbb{Z}_{\geq 0}),
$$
and the operator $\Delta: \bar{\Lambda}[x]\to\bar{\Lambda}[x]$ defined by
$$
\Delta(x^rf) = \Delta(x^r)f,\ \ \Delta(1) = 0, \ \ \Delta(x^r) = \sum_{s=0}^{r-1}x^{r-1-s}p_s - rx^{r-1}\ \ (f\in\bar{\Lambda}, r\in\mathbb{Z}_{\geq 0}).
$$
The above definition is motivated by the fact that the homomorphism $\bar{\Lambda}[x]\to \Lambda_N[x_i]$ that maps $p_r\mapsto x_1^r+\cdots+x_N^r$ ($r\in\mathbb{Z}_{\geq 0}$) and $x\mapsto x_i$ ($i=1,\ldots,N$) intertwines $D_\infty$ and $D_{i,N}$.

Introducing also the linear `symmetrisation' operator $E:\bar{\Lambda}[x]\to\bar{\Lambda}$ by
$$
E(x^rf) = E(x^r)f,\ \ E(1) = 1,\ \ E(x^r) = p_r,\ \ (f\in\bar{\Lambda}, r\in\mathbb{Z}_{\geq 0}),
$$
rational Calogero--Moser--Sutherland integrals at infinity $L^{(r)}:\bar{\Lambda}\to\bar{\Lambda}$ ($r\in\mathbb{Z}_{\geq 0}$) are given by
\begin{equation}
\label{Lr}
L^{(r)} = \mathrm{Res}\, E\circ D_\infty^r,
\end{equation}
with $\mathrm{Res}$ denoting restriction to $\bar{\Lambda}$. Combining the intertwining relation between $D_\infty$ and $D_{i,N}$ with \eqref{Lp}, it is readily seen that the diagram 
\begin{equation}
\label{LNDiag}
\begin{CD}
\bar{\Lambda} @>L^{(r)}>> \bar{\Lambda}\\
@VV\varphi_NV @VV\varphi_NV\\
\Lambda_N @>L_N^{(r)}>> \Lambda_N
\end{CD}
\end{equation}
is commutative for all $r\in\mathbb{Z}_{\geq 0}$; and, as a straightforward consequence, it follows that $[L^{(r)},L^{(s)}]=0$ ($r,s\in\mathbb{Z}_{\geq 0}$).

\subsection{Deformed rational Calogero--Moser--Sutherland operators}
By setting
\begin{equation}
\label{xExt}
x_{n+i} = y_i\ \ (i=1,\ldots,m)
\end{equation}
and using the `parity' function
\begin{equation}
\label{parity}
p(i) :=
\left\{
\begin{array}{ll}
0, & i = 1,\ldots,n\\
1, & i = n+1,\ldots,n+m
\end{array}
\right.
\end{equation}
the rational limit of \eqref{cLnm} takes the simple and convenient form
\begin{multline}
\label{Lnm}
L_{n,m} := \sum_{i=1}^{n+m}(-\theta)^{p(i)}\frac{\partial^2}{\partial x_i^2}\\
- 2\sum_{1\leq i<j\leq n+m}\frac{(-\theta)^{1-p(i)-p(j)}}{x_i-x_j}\left((-\theta)^{p(i)}\frac{\partial}{\partial x_i}-(-\theta)^{p(j)}\frac{\partial}{\partial x_j}\right).
\end{multline}

We shall make use of a recursive formula for its quantum integrals and a connection with Calogero--Moser--Sutherland operators at infinity, both due to Sergeev and Veselov \cite{SV04,SV15}.

Specifically, taking $\partial_i^{(1)}=(-\theta)^{p(i)}\frac{\partial}{\partial x_i}$, differential operators of order $r>1$ are defined recursively by
\begin{equation}
\label{pari}
\partial_i^{(r)} = \partial_i^{(1)}\partial_i^{(r-1)}-\sum_{j\neq i}\frac{(-\theta)^{1-p(j)}}{x_i-x_j}\big(\partial_i^{(r-1)}-\partial_j^{(r-1)}\big)
\end{equation}
and the quantum integrals are given by
\begin{equation}
\label{Lrnm}
L^{(r)}_{n,m} = \sum_{i=1}^{n+m} (-\theta)^{-p(i)}\partial_i^{(r)}\ \ (r\in\mathbb{N}),
\end{equation}
with $L^{(2)}_{n,m}=L_{n,m}$.

Moreover, substituting $\mathbb{Z}_{\geq 0}$ for $\mathbb{N}$ in \eqref{varphinm} produces a homomorphism $\varphi_{n,m}:\bar{\Lambda}\to\Lambda_{n,m}$ mapping $p_0\mapsto\varphi_{n,m}(p_0):=n-m/\theta$. From Thm.~3.3 in \cite{SV15}, we recall that the diagram
\begin{equation}
\label{LnmDiag}
\begin{CD}
\bar{\Lambda} @>L^{(r)}>> \bar{\Lambda}\\
@VV\varphi_{n,m}V @VV\varphi_{n,m}V\\
\Lambda_{n,m} @>L_{n,m}^{(r)}>> \Lambda_{n,m}
\end{CD}
\end{equation}
is commutative for all $r\in\mathbb{N}$ and commutativity of the deformed Calogero--Moser--Sutherland operators \eqref{Lrnm}, i.e.
$$
\big[L_{n,m}^{(r)},L_{n,m}^{(s)}\big] = 0\ \ (r,s\in\mathbb{N}),
$$
follows from that of the operators \eqref{Lr}.

We proceed to describe a Harish-Chandra type isomorphism mapping the algebra of quantum integrals
\begin{equation}
\label{Qnm}
Q_{n,m} := \mathbb{C}\big[L_{n,m}^{(1)},L_{n,m}^{(2)},\ldots\big]
\end{equation}
of the deformed rational Calogero--Moser--Sutherland system onto $\Lambda_{n,m}$. That such an identification of $Q_{n,m}$ with $\Lambda_{n,m}$ is possible was first observed by Sergeev and Veselov \cite{SV04}.

We write $V_{n,m}$ for $\mathbb{C}^{n+m}$ equipped with the bilinear form
$$
B_{n,m}(u,v) := \sum_{i=1}^n u_iv_i - \theta\sum_{i=1}^m u_{n+i}v_{n+i}
$$
and $\mathscr{D}_{n,m}$ for the algebra of differential operators with constant (complex) coefficients generated by $\frac{\partial}{\partial x_i}$ ($i=1,\ldots,n$) and $\frac{\partial}{\partial y_i}$ ($i=1,\ldots,m$). Given $p\in P_{n,m}$, we let
$$
\partial(p) = p\left(\frac{\partial}{\partial x_1},\ldots,\frac{\partial}{\partial x_n},-\theta\frac{\partial}{\partial y_1},\ldots,-\theta\frac{\partial}{\partial y_m}\right).
$$
Identifying $P_{n,m}$ with $S(V_{n,m})$ and $\mathscr{D}_{n,m}$ with $S(V_{n,m}^*)$, the map $p\mapsto\partial(p)$ amounts to the isomorphism $S(V_{n,m})\to S(V_{n,m}^*)$ induced by $B_{n,m}$. (Here, we have used the notation $S(V)$ for the symmetric algebra of $V$.)

Letting $\mathscr{R}_{n,m}$ be the algebra of rational functions generated by $(x_i-x_j)^{-1}$ ($1\leq i<j\leq n$), $(x_i-y_j)^{-1}$ ($1\leq i\leq n$, $1\leq j\leq m$) and $(y_i-y_j)^{-1}$ ($1\leq i<j\leq m$), we introduce the algebra of differential operators $\mathscr{D}_{n,m}[\mathscr{R}_{n,m}]$, generated by the derivatives $\frac{\partial}{\partial x_i}$ and $\frac{\partial}{\partial y_i}$ with coefficients in $\mathscr{R}_{n,m}$, and note that $Q_{n,m}\subset\mathscr{D}_{n,m}[\mathscr{R}_{n,m}]$.

For $N\in\mathbb{N}$, we write $\mathfrak{N}(N,\mathbb{Z}_{\geq 0})$ for the set of strictly upper triangular $N\times N$ matrices with non-negative integer entries. Using again the convention $x_{n+i}=y_i$ ($i=1,\ldots,m$), we observe that any $L\in\mathscr{D}_{n,m}[\mathscr{R}_{n,m}]$ has a representation of the form
$$
L=\sum_{M\in \mathfrak{N}(n+m,\mathbb{Z}_{\geq 0})} \prod_{1\leq i<j\leq n+m}(x_i-x_j)^{-M_{ij}}\cdot \partial(\ell_M)
$$
for some $\ell_M\in P_{n,m}$. Hence, we can define a homomorphism $\psi_{n,m}:\mathscr{D}_{n,m}[\mathscr{R}_{n,m}]\to P_{n,m}$ by setting $\psi_{n,m}(L)=\ell_0$. From \eqref{pari}--\eqref{Lrnm}, it is readily inferred that
\begin{equation}
\label{psiAct}
\psi_{n,m}\big(L_{n,m}^{(r)}\big) = \sum_{i=1}^{n+m} (-\theta)^{-p(i)}x_i^r = p_{r,\theta}(x,y).
\end{equation}
Since the deformed power sums $p_{r,\theta}$ generate $\Lambda_{n,m}$, it follows that $\psi_{n,m}$ maps $Q_{n,m}$ onto $\Lambda_{n,m}$. In Section \ref{Sec:hypg3}, we shall provide a rather different description of this map and, as a consequence, infer injectivity.

\section{Generalised hypergeometric series}
\label{Sec:Hyperg}
Our main results rely on generalised hypergeometric series in two sequences of variables associated with either Jack polynomials, Jack symmetric functions or super-Jack polynomials.

\subsection{Associated with Jack polynomials}
In the case of Jack polynomials in $N$ variables $x=(x_1,\ldots,x_N)$, such series appeared, in particular, in Macdonald's widely circulated informal working paper \cite{Mac13}. They take a particularly simple form when expressed in terms of Kaneko's normalisation \cite{Kan93}
\begin{equation}
\label{Kan}
C_\lambda^{(\theta)}(x) = \frac{|\lambda|!}{\prod_{s\in\lambda}(a(s)+1+\theta l(s))} P_\lambda^{(\theta)}(x),
\end{equation}
also characterised by 
$$
\sum_{|\lambda|=k}C_\lambda^{(\theta)}(x) = p_1(x)^k\ \ \ (k\in\mathbb{Z}_{\geq 0}).
$$
Indeed, the simplest such series, with no additional parameters beyond $\theta$ and the only one we require, is given by
\begin{equation}
\label{FN}
F^{(\theta)}_N(x,y) = \sum_{d=0}^\infty F^{(\theta)}_{N,d}(x,y),\ \ \ F^{(\theta)}_{N,d}(x,y) = \sum_{|\lambda|=d} \frac{1}{|\lambda|!}\frac{C^{(\theta)}_\lambda(x)C^{(\theta)}_\lambda(y)}{C^{(\theta)}_\lambda(1^N)},
\end{equation}
where $F^{(\theta)}_{N,d}(x,y)$ is the homogeneous part of degree $d$ in both the $x_i$ and the $y_j$. (In \cite{Mac13}, this series is denoted ${}_0F_0(x,y;1/\theta)$). From \eqref{spec} and \eqref{Kan}, it is clear that the individual terms in these series are well-defined whenever $\theta$ is not a negative rational number or zero. Moreover, for $\theta>0$, the infinite series is known to converge locally uniformly on $\mathbb{C}^N\times\mathbb{C}^N$ and therefore define an entire function; see e.g.~Props.~3.10--11 in \cite{BF98} or Thm.~6.5 in \cite{BR23}.

As indicated in \cite{OO97} (see also \cite{BR23} for a more detailed explanation), the generalised hypergeometric series $F^{(\theta)}_N(x,y)$ is equal to the $A_{N-1}$ instance of Opdam's \cite{Opd93} multivariable Bessel functions associated with root systems. In particular, this equality manifests itself in the following joint eigenfunction property.

\begin{proposition}
\label{Prop:FN}
Assume that $\theta$ is not a negative rational number or zero. For each $p\in\Lambda_N$, we have
$$
L_{p,N}(x)F_N^{(\theta)}(x,y) = p(y)F_N^{(\theta)}(x,y).
$$
\end{proposition}

When $p=p_r$ with $r=1,2$, this result can be found in \cite{BF97}, whereas an eigenfunction property equivalent to the general case is sketched in \cite{OO97}. For arbitrary $p\in\Lambda_N$, a proof of the proposition is readily inferred from results by Baker and Forrester \cite{BF98} on a non-symmetric generalised hypergeometric series $\mathscr{K}_A(x,y)$, defined in Eq.~(3.17) in \cite{BF98}. More specifically, from its joint eigenfunction property
$$
D_{i,N}(x)\mathscr{K}_A(x,y) = y_i\mathscr{K}_A(x,y)\ \ \ (i = 1,\ldots,N)
$$
and symmetrisation property
$$
\sum_{\sigma\in S_N}\mathscr{K}_A(x,\sigma y) = N! F_N^{(\theta)}(x,y),
$$
obtained in Thm.~3.8(c) and Prop.~3.11 in \cite{BF98}, respectively, it follows that
\begin{equation*}
\begin{split}
L_p(x)F_N^{(\theta)}(x,y) &= p\big(D_{1,N}(x),\ldots,D_{N,N}(x)\big)F_N^{(\theta)}(x,y)\\
&= \frac{1}{N!}\sum_{\sigma\in S_N}p\big(D_{1,N}(x),\ldots,D_{N,N}(x)\big)\mathscr{K}_A(x,\sigma y)\\
&= \frac{1}{N!}\sum_{\sigma\in S_N}p(y_{\sigma^{-1}(1)},\ldots,y_{\sigma^{-1}(N)})\mathscr{K}_A(x,\sigma y)\\
&= p(y_1,\ldots,y_N)F_N^{(\theta)}(x,y).
\end{split}
\end{equation*}

In terms of homogeneous components, the result reads as follows.

\begin{corollary}
\label{Cor:FNd}
For all $d,k\in\mathbb{Z}_{\geq 0}$ and $p\in\Lambda_{N}^k$, we have
$$
L_p(x)F^{(\theta)}_{N,d}(x,y) = p(y)F^{(\theta)}_{N,d-k}(x,y),
$$
with $F^{(\theta)}_{N,d-k}\equiv 0$ when $d<k$ and the above assumptions on $\theta$ in place.
\end{corollary}

\subsection{Associated with Jack symmetric functions}
In the context of symmetric functions, generalised hypergeometric series were introduced in \cite{DH12}. For our purposes, it suffices to consider the simplest instance, involving only the parameters $\theta$ and $p_0$, and recall its homogeneous components:
\begin{equation}
\label{Finf}
F^{(\theta,p_0)}_d(x,y) = \sum_{|\lambda|=d} \frac{1}{|\lambda|!}\frac{C^{(\theta)}_\lambda(x)C^{(\theta)}_\lambda(y)}{\epsilon_{p_0}\big(C^{(\theta)}_\lambda\big)}\ \ \ (d\in\mathbb{Z}_{\geq 0}),
\end{equation}
where, as before, $\epsilon_{p_0}$ denotes the homomorphism $\Lambda\to\mathbb{C}[p_0]$ given by $p_r\mapsto p_0$ ($r\in\mathbb{N}$). Just as in the previous section, $\theta$ not being a negative rational number or zero ensures that this series is well-defined.

The following infinite-dimensional generalisation of Cor.~\ref{Cor:FNd} is now readily identified and proved.

\begin{proposition}
\label{Prop:LrFd}
Take $\theta\in\mathbb{C}$ not of the form \eqref{thetaValsPrel}. For all $r,d\in\mathbb{Z}_{\geq 0}$, we have
\begin{equation}
\label{FEq}
L^{(r)}(x)F^{(\theta,p_0)}_{d}(x,y) = p_r(y)F^{(\theta,p_0)}_{d-r}(x,y),
\end{equation}
with $F^{(\theta,p_0)}_{d-r}\equiv 0$ when $d<r$.
\end{proposition}

\begin{proof}
Since $L^{(r)}$ preserves $\bar{\Lambda}$ and lowers the degree by $r$, it is clear that the left-hand side of \eqref{FEq} is identically zero whenever $r>d$. Hence, fixing $r,d\in\mathbb{Z}_{\geq 0}$ such that $r\leq d$, we consider the symmetric function
$$
f(x,y) := L^{(r)}(x)F^{(\theta,p_0)}_{d}(x,y) - p_r(y)F^{(\theta,p_0)}_{d-r}(x,y),
$$
which amounts to a polynomial in $p_r(x), p_r(y)$, $1\leq r\leq d$, with coefficients depending rationally on $p_0$.

At this point, we choose $N\geq d$. We note that \eqref{spec} and $l^\prime(s)=i-1<\ell(\lambda)\leq|\lambda|$, where $s=(i,j)\in\lambda$, ensures that $f(x,y)$ has no pole at $p_0=N$. From \eqref{LNDiag} and Cor.~\ref{Cor:FNd}, we can thus infer that $\varphi_N^{(x)}\varphi_N^{(y)}f(x,y)=0$. (Of course $F^{(\theta,p_0)}_d(x,y)\notin\bar{\Lambda}$, but this minor technical snag is easily resolved by clearing denominators or, as discussed in Section 2.4 in \cite{DH12}, extending the definition of $\varphi_N$ to all elements in $\mathbb{C}(p_0)\otimes\Lambda$ that lack a pole at $p_0=N$.)  By algebraic independence of the $\varphi_N(p_r)$, $1\leq r\leq d$, it follows that each coefficient of $f(x,y)$ must vanish at $p_0=N$. Since this is the case for all $N\geq d$, the coefficients vanish identically and $f(x,y)\equiv 0$.
\end{proof}

\subsection{Associated with super-Jack polynomials}
\label{Sec:hypg3}
In this section, we work with sequences of $n+m$ variables $x=(x_1,\ldots,x_n)$ and $y=(y_1,\ldots,y_m)$ as well as $z=(z_1,\ldots,z_n)$ and $w=(w_1,\ldots,w_m)$. Introducing the renormalised super-Jack polynomials
\begin{equation}
\label{SC}
SC^{(\theta)}_\lambda(x,y) := \frac{|\lambda|!}{\prod_{s\in\lambda}(a(s)+1+\theta l(s))} SP_\lambda^{(\theta)}(x,y),
\end{equation}
we perform the substitutions $p_0\mapsto n-m/\theta$, $C^{(\theta)}_\lambda(x)\mapsto SC^{(\theta)}_\lambda(x,y)$ and $C^{(\theta)}_\lambda(y)\mapsto SC^{(\theta)}_\lambda(z,w)$ in \eqref{Finf} to obtain
\begin{equation}
\label{SFd}
SF^{(\theta)}_{n,m;d}(x,y;z,w) := \sum_{\substack{\lambda\in H_{n,m}\\ |\lambda|=d}} \frac{1}{|\lambda|!}\frac{SC^{(\theta)}_\lambda(x,y)SC^{(\theta)}_\lambda(z,w)}{SC^{(\theta)}_\lambda(1^{n+m})}\ \ \ (d\in\mathbb{Z}_{\geq 0}).
\end{equation}

Before proceeding further, a few remarks are in order: First, the above substitutions are essentially given by the homomorphism $\varphi_{n,m}$ (cf.~the remark in the proof of Prop.~\ref{Prop:LrFd}); second, summation over $d\in\mathbb{Z}_{\geq 0}$ yields
\begin{equation}
\label{SF}
\begin{split}
SF^{(\theta)}_{n,m}(x,y;z,w) &:= \sum_{d=0}^\infty SF^{(\theta)}_{n,m;d}(x,y;z,w)\\
&= \sum_{\lambda\in H_{n,m}}\frac{1}{|\lambda|!}\frac{SC^{(\theta)}_\lambda(x,y)SC^{(\theta)}_\lambda(z,w)}{SC^{(\theta)}_\lambda(1^{n+m})},
\end{split}
\end{equation}
which is identical with ${}_0\mathscr{SF}_0(x,y;z,w)$ in Section 6.3 of \cite{DH12}; and, third, to ensure that each $SF^{(\theta)}_{n,m;d}(x,y;z,w)$ is well-defined, we need to avoid the parameter values
\begin{equation}
\label{thetaVals}
\theta = \frac{i}{j},\ \ \ i\in-\mathbb{Z}_{\geq 0},\, j\in\mathbb{N}\ \text{or}\ 1\leq i\leq m,\, 1\leq j\leq n,
\end{equation}
since we might encounter a pole of a coefficient in $SC^{(\theta)}_\lambda$ when $\theta$ is a negative rational number and $SC^{(\theta)}_\lambda(1^{n+m})$ could vanish when $\theta$ takes one of the finitely many positive rational values specified above.

The methods used in \cite{BF98,BR23} to study convergence properties of $F_N^{(\theta)}$ do not directly apply to $SF^{(\theta)}_{n,m}$. However, using an approach from Desrosiers and Liu \cite{DL15}, it is possible to establish convergence for generic $\theta>0$. A precise statement and proof can be found in Appendix \ref{App:Conv}. For the remaining parameter values, we offer two interpretations: Either view $SF^{(\theta)}_{n,m}$ as a formal power series or consider each homogeneous component separately, so that convergence is not an issue.

By combining Prop.~\ref{Prop:LrFd} with \eqref{LnmDiag}, we establish
\begin{equation*}
L^{(r)}_{n,m}(x,y)SF^{(\theta)}_{n,m;d}(x,y;z,w) = p_{r,\theta}(z,w)SF^{(\theta)}_{n,m;d-r}(x,y;z,w)\ \ (r,d\in\mathbb{Z}_{\geq 0}),
\end{equation*}
where $SF^{(\theta)}_{d-r}\equiv 0$ if $d<r$. As a direct consequence of \eqref{Qnm}, \eqref{psiAct} and \eqref{SF}, we thus obtain the following lemma.

\begin{lemma}
\label{Lemma:SFnm}
For $\theta\in\mathbb{C}$ not of the form \eqref{thetaVals} and $L\in Q_{n,m}$, we have
$$
L(x,y)SF^{(\theta)}_{n,m}(x,y;z,w) = \psi_{n,m}(L)(z,w)SF^{(\theta)}_{n,m}(x,y;z,w).
$$
\end{lemma}

Using this joint eigenfunction property, it is straightforward to verify that the Harish-Chandra homomorphism $\psi_{n,m}$ maps $Q_{n,m}$ {\em one-to-one} onto $\Lambda_{n,m}$. To this end, let $L\in Q_{n,m}$ be such that $\psi_{n,m}(L)=0$. Since the $SP_\lambda^{(\theta)}$ ($\lambda\in H_{n,m}$) span $\Lambda_{n,m}$, it follows that $L$ vanishes on $\Lambda_{n,m}$. As long as the Vandermonde polynomial
$$
V_{n+m}(x,y) := \prod_{1\leq i<j\leq n}(x_i-x_j)\cdot \prod_{1\leq i<j\leq m}(y_i-y_j)\cdot \prod_{i=1}^n\prod_{i=1}^m(x_i-y_j)\neq 0,
$$
we can use $f_r:=p_{r,\theta}-p_{r,\theta}(x,y)$ with $r=1,\ldots,n+m$ as coordinate functions (centered) at $(x,y)\in\mathbb{C}^n\times\mathbb{C}^m$. (This is easily seen by computing the Jacobian determinant $\partial(p_{1,\theta},\ldots,p_{n+m,\theta})/\partial(x_1,\ldots,x_n,y_1,\ldots,y_m)$.) Then, we can write
\begin{equation}
\label{LExp}
L = \sum c_\alpha \partial_{f_1}^{\alpha_1}\cdots\partial_{f_{n+m}}^{\alpha_{n+m}},
\end{equation}
where the coefficient functions $c_\alpha\neq 0$ only for a finite subset of multi-indices $\alpha=(\alpha_1,\ldots,\alpha_{n+m})\in\mathbb{Z}_{\geq 0}^{n+m}$. Now, assume that $L$ is a non-trivial differential operator. Then, there exists $\alpha\in\mathbb{Z}_{\geq 0}^{n+m}$ of minimal weight $|\alpha|=\alpha_1+\cdots+\alpha_{n+m}$ such that $c_\alpha\neq 0$. However, since $L$ amounts to the zero operator on $\Lambda_{n,m}$, we have
$$
0 = L f_1^{\alpha_1}\cdots f_{n+m}^{\alpha_{n+m}} = c_\alpha \alpha!,
$$
which contradicts the above assumption.

We can thus conclude that $\psi_{n,m}:Q_{n,m}\to\Lambda_{n,m}$ is an isomorphism, write
\begin{equation}
L_p = \psi_{n,m}^{-1}(p)\ \ (p\in\Lambda_{n,m})
\end{equation}
and reformulate Lemma \ref{Lemma:SFnm} as the following 'deformed' analogue of Prop.~\ref{Prop:FN}.

\begin{proposition}
\label{Prop:SFnm}
For each $p\in\Lambda_{n,m}$, we have
$$
L_p(x,y)SF^{(\theta)}_{n,m}(x,y;z,w) = p(z,w)SF^{(\theta)}_{n,m}(x,y;z,w),
$$
as long as the $\theta$-values \eqref{thetaVals} are avoided.
\end{proposition}

Just as in the Jack polynomial case (cf.~Cor.~\ref{Cor:FNd}), the corresponding result for the homogenous components $SF_{n,m;d}^{(\theta)}$ immediately follows.

\begin{corollary}
\label{Cor:SFnmd}
For all $d,k\in\mathbb{Z}_{\geq 0}$ and $p\in\Lambda_{n,m}^k$, we have
$$
L_p(x,y)SF^{(\theta)}_{n,m;d}(x,y;z,w) = p(z,w)F^{(\theta)}_{n,m;d-k}(x,y;z,w),
$$
with $SF^{(\theta)}_{n,m;d-k}\equiv 0$ when $d<k$ and $\theta$ not of the form \eqref{thetaVals}.
\end{corollary}

\section{The bilinear form}
\label{Sec:Form}

We are now ready to introduce the relevant bilinear form on $\Lambda_{n,m}$.

\begin{definition}
\label{Def:Form}
Assuming that $\theta\in\mathbb{C}$ is not a negative rational number or zero, we define a (complex) bilinear form on $\Lambda_{n,m}$ by
$$
(p,q)_{n,m} = (L_pq)(0)\ \ (p,q\in\Lambda_{n,m}).
$$
\end{definition}

Writing $L^*$ for the adjoint of $L\in Q_{n,m}$, we proceed to formulate and prove some basic properties of $(\cdot,\cdot)_{n,m}$.

\begin{proposition}
\label{Prop:Form}
Excluding the values of $\theta$ given in \eqref{thetaVals}, we have
\begin{enumerate}
\item $(p,q)_{n,m}=L_p(x,y)L_q(z,w)SF^{(\theta)}_{n,m}(x,y;z,w)\arrowvert_{x=y=z=w=0}$,
\item $(p,q)_{n,m}=0$ whenever $p,q\in\Lambda_{n,m}$ are homogenous of different degrees,
\item $(\cdot,\cdot)_{n,m}$ is symmetric,
\item $L_p^*=p$.
\end{enumerate}
\end{proposition}

\begin{proof}
\begin{enumerate}
\item By degree considerations, it is readily seen that the polynomial $L_p(x,y)L_q(z,w)SF^{(\theta)}_{n,m;d}(x,y;z,w)$ vanishes at $x=y=z=w=0$ unless $\deg p=\deg q=d$, in which case Cor.~\ref{Cor:SFnmd} entails
\begin{equation*}
\begin{split}
L_p(x,y)L_q(z,w)SF^{(\theta)}_{n,m;d}(x,y;z,w) &= L_p(x,y)q(x,y)SF^{(\theta)}_{n,m;0}\\
&= (L_pq)(0)SF^{(\theta)}_{n,m;0}
\end{split}
\end{equation*}
and, since $SF^{(\theta)}_{n,m;0}=1$, the claim follows.

\item Follows immediately from the definition of $(\cdot,\cdot)_{n,m}$ and the fact that $L_p$ is homogenous of degree $-\deg p$.

\item
Clear from \eqref{SF} and Property (1).

\item By definition, we have
$$
(s,L_pq)_{n,m} = (L_sL_pq)(0) = (L_{sp}q)(0) = (ps,q)_{n,m}
$$
for all $p,q,s\in\Lambda_{n,m}$.
\end{enumerate}
\end{proof}

As is clear from (1) in Prop.~\ref{Prop:Form}, there is close connection between the generalised hypergeometric series $SF^{(\theta)}_{n,m}$ and the bilinear form $(\cdot,\cdot)_{n,m}$. This connection if further clarified in the following proposition, which identifies $SF^{(\theta)}_{n,m}$ as the reproducing kernel of $(\cdot,\cdot)_{n,m}$.

\begin{proposition}
\label{Prop:ReprKer}
For each $p\in\Lambda_{n,m}$, we have
\begin{equation}
\label{reprEq}
\big(p(x,y),SF^{(\theta)}_{n,m}(x,y;z,w)\big)_{n,m} = p(z,w),
\end{equation}
where we assume that $\theta$ is not of the form \eqref{thetaVals}.
\end{proposition}

\begin{proof}
The result is an immediate consequence of Def.~\ref{Def:Form}, Prop.~\ref{Prop:SFnm} and the observation that
$$
SF^{(\theta)}_{n,m}(0^n,0^m;z,w) = SF^{(\theta)}_{n,m}(x,y;0^n,0^m) = SF^{(\theta)}_{n,m}(0^n,0^m;0^n,0^m) = 1,
$$
which, in turn, is clear from \eqref{SF}.
\end{proof}

\begin{corollary}
Assuming that $\theta$ is not of the form \eqref{thetaVals}, the bilinear form $(\cdot,\cdot)_{n,m}$ is nondegenerate.
\end{corollary}

We continue to make precise the integral representation \eqref{intRep} for the bilinear form $(\cdot,\cdot)_{n,m}$. First of all, we note that the operator $e^{-L_{n,m}/2}$ has a well defined action on $\Lambda_{n,m}$. Indeed, since $L_{n,m}$ is homogeneous of degree $-2$ on $\Lambda_{n,m}$, it is locally nilpotent, and so if $p\in\Lambda_{n,m}$ has degree $d\in\mathbb{Z}_{\geq 0}$, then
$$
e^{-L_{n,m}/2}p := \sum_{k=0}^{\lfloor d/2\rfloor} \frac{(-1)^k}{2^k k!}L_{n,m}^k p.
$$
Requiring that $\xi\in\mathbb{R}^n$ and $\eta\in\mathbb{R}^m$ satisfy
\begin{equation}
\label{xieta}
\xi_n>\cdots>\xi_1,\ \ \ \eta_m>\cdots>\eta_1,\ \ \ \xi_i\neq\eta_j\ \ (1\leq i\leq n, 1\leq j\leq m),
\end{equation}
so that we are working with $\mathrm{Im}(x_i-x_j)<0$ ($1\leq i<j\leq n$) and $\mathrm{Im}(y_i-y_j)<0$ ($1\leq i<j\leq m$), we fix the branch of $A_{n,m}(x,y)$ by taking the principal value of $z^\rho$ for $z\in\mathbb{C}^*$ and $\rho\in\mathbb{C}$. From Def.~\ref{Def:Form}, it is clear that $(1,1)_{n,m}=1$, which requires
$$
M_{n,m} = \int_{\mathbb{R}^n+i\xi}\int_{\mathbb{R}^m+i\eta} \frac{e^{-x^2/2+\theta^{-1}y^2/2}}{A_{n,m}(x,y)}dxdy.
$$
This generalised Macdonald--Mehta integral was computed in \cite{FHV13}. Specifically, taking $t_i\to t_i/\sqrt{-\rho}$ and setting $\rho=\theta$ in Eq.~(33), we infer from Prop.~6.1 that
\begin{equation}
\label{MnmExpr}
M_{n,m} = C_{n,m}
\prod_{i=1}^n\prod_{j=1}^m \frac{1}{j-i\theta}\cdot \prod_{i=1}^n \frac{\Gamma(1-\theta)}{\Gamma(1-i\theta)}\cdot \prod_{j=1}^m \frac{\Gamma(1-1/\theta)}{\Gamma(1-j/\theta)},
\end{equation}
where
\begin{equation}
\label{Cnm}
C_{n,m}=(2\pi)^{\frac{n+m}{2}}(-\theta)^{n/2+n(n-1)/(2\theta)}\exp\left(-i\pi\left(\frac{n(n-1)\theta}{2}+\frac{m(m-1)}{2\theta}\right)\right).
\end{equation}

For suitable values of $\theta$, we can now establish the validity of \eqref{intRep}.

\begin{proposition}
\label{Prop:intRep}
Assuming that $\mathrm{Re}\,\theta<0$ and $\theta$ is not a negative rational number, the integral representation \eqref{intRep}, with $M_{n,m}$ given by \eqref{MnmExpr}--\eqref{Cnm}, for the bilinear form $(\cdot,\cdot)_{n,m}$ holds true as long as $\xi\in\mathbb{R}^n$ and $\eta\in\mathbb{R}^m$ satisfy the restrictions in \eqref{xieta}.
\end{proposition}

\begin{proof}
Under our assumptions on $\theta$ and $\xi,\eta$, it is clear that the integral in the right-hand side of \eqref{intRep} is convergent

We note that the bilinear form $(\cdot,\cdot)_{n,m}$ is uniquely determined by the following properties:
\begin{enumerate}
\item $(1,1)_{n,m}=1$,
\item $(p,q)_{n,m}=(q,p)_{n,m}$,
\item $(ps,q)_{n,m}=(s,L_pq)_{n,m}$,
\end{enumerate}
for arbitrary $p,q,s\in\Lambda_{n,m}$. Indeed, by linearity and Property (2), it suffices to consider homogenous $p,q\in\Lambda_{n,m}$ such that $\deg p\geq\deg q$, and, from Property (3), we get
$$
(p,q)_{n,m} = (1,L_pq)_{n,m}.
$$
If $\deg p>\deg q$, the right-hand side is clearly zero, and in the remaining case $\deg p=\deg q$, it follows from Property (1) that
$$
(1,L_pq)_{n,m}=(L_pq)(0)(1,1)_{n,m} = (L_pq)(0),
$$
(where the first equality is due to $\deg(L_pq)=0$). Hence, writing $(p,q)^\prime_{n,m}$ for the right-hand side of \eqref{intRep}, it suffices to establish Properties (1)-(3) for the resulting bilinear form $(\cdot,\cdot)^\prime_{n,m}$ on $\Lambda_{n,m}$.

Properties (1) and (2) are obvious.

To prove Property (3), we rewrite the weight function in terms of the convenient notation \eqref{xExt}--\eqref{parity}:
$$
e^{-\frac{1}{2}\sum_{i=1}^{n+m}(-\theta)^{-p(i)}x_i^2}\cdot \prod_{1\leq i<j\leq n+m}(x_i-x_j)^{-2(-\theta)^{1-p(i)-p(j)}}.
$$
Then, we use \eqref{pari}--\eqref{Lrnm} to verify, by direct computations, that the corresponding formal adjoint of $L^{(r)}_{n,m}$ is given by
\begin{equation*}
\widehat L^{(r)}_{n,m} = \sum_{i=1}^{n+m} (-\theta)^{-p(i)}\widehat\partial_i^{(r)},
\end{equation*}
with $\widehat\partial_i^{(1)} = x_i-\partial_i^{(1)}$ and
\begin{equation}
\label{hatPari}
\partial_i^{(r)} = \widehat\partial_i^{(1)}\widehat\partial_i^{(r-1)}+\sum_{j\neq i}\frac{(-\theta)^{1-p(j)}}{x_i-x_j}\big(\widehat\partial_i^{(r-1)}-\widehat\partial_j^{(r-1)}\big)
\end{equation}
for $r>1$.

Since $L_{p_{r,\theta}}=L^{(r)}_{n,m}$ and the deformed power sums $p_{r,\theta}(x,y)$ generate $\Lambda_{n,m}$, the desired Property (3) will follow once we prove that
\begin{equation}
\label{hatLnmConj}
e^{L_{n,m}/2}\widehat L^{(r)}_{n,m}e^{-L_{n,m}/2} = p_{r,\theta}(x,y).
\end{equation}
Following the approach in Section \ref{Sec:Hyperg}, we first establish such a conjugation formula for the operators
$$
\widehat L_{p,N} = \mathrm{Res}\big(p\big(x_1-D_{1,N},\ldots,x_N-D_{N,N}\big)\big)\ \ \ (p\in\Lambda_N)
$$
in $\Lambda_N$, then lift it to $\Lambda$ and finally restrict the result to $\Lambda_{n,m}$.

Using the commutation relations
$$
[D_{i,N},x_j] =
\left\{
\begin{array}{ll}
1+\theta\sum_{k\neq i}\sigma_{ik}, & j = i\\
-\theta\sigma_{ij}, & j\neq i
\end{array}
\right.
$$
a straightforward computation yields
$$
\left[\sum_{i=1}^N D_{i,N}^2,x_j\right] = 2D_{j,N},
$$
which entails
$$
e^{\frac{1}{2}\sum_{i=1}^ND_{i,N}^2}x_je^{-\frac{1}{2}\sum_{i=1}^ND_{i,N}^2} = e^{\frac{1}{2}\mathrm{ad}\left(\sum_{i=1}^ND_{i,N}^2\right)}x_j = x_j+D_{j,N},
$$
where we have used the standard formula $\mathrm{Ad}_{e^X} = e^{\mathrm{ad}_X}$. It follows that
\begin{multline}
\label{hatLpNConj}
e^{L_N/2}\widehat L_{p,N}e^{-L_N/2}\\
= \mathrm{Res}\left(e^{\frac{1}{2}\sum_{i=1}^ND_{i,N}^2}p\big(x_1-D_{1,N},\ldots,x_N-D_{N,N}\big)e^{-\frac{1}{2}\sum_{i=1}^ND_{i,N}^2}\right)\\
= p(x_1,\ldots,x_N).
\end{multline}
Consulting the proof of Thm.~2.2 in \cite{SV15}, it is readily seen how to lift the operators $\widehat L_N^{(r)}:=\widehat L_{p_r,N}$ ($r\in\mathbb{Z}_{\geq 0}$) to $\bar{\Lambda}$. Specifically, introducing the operators
$$
\widehat L^{(r)} := \mathrm{Res}\, E\circ (x-D_\infty)^r : \bar{\Lambda}\to \bar{\Lambda}\ \ \ (r\in\mathbb{N}),
$$
we have the commutative diagram
\begin{equation}
\label{hatLNDiag}
\begin{CD}
\bar{\Lambda} @>\widehat L^{(r)}>> \bar{\Lambda}\\
@V\varphi_NVV @VV\varphi_NV\\
\Lambda_N @>\widehat L_N^{(r)}>> \Lambda_N
\end{CD}
\end{equation}
for each $r\in\mathbb{N}$. Since $f\in\bar{\Lambda}$ satisfies $\varphi_N(f)=0$ for all $N\in\mathbb{N}$ if and only if $f\equiv 0$ (cf.~Lemma 2.3 in \cite{SV15}), we infer from \eqref{LNDiag} and \eqref{hatLpNConj}--\eqref{hatLNDiag} that
$$
e^{L^{(2)}/2}\widehat L^{(r)}e^{-L^{(2)}/2} = p_r\ \ (r\in\mathbb{Z}_{\geq 0}).
$$
Finally, by easily identifiable modifications of the discussion in Section 3 of \cite{SV15}, we find that the diagram
\begin{equation}
\label{LDiag}
\begin{CD}
\bar{\Lambda} @>\widehat L^{(r)}>> \bar{\Lambda}\\
@V\varphi_{n,m}VV @VV\varphi_{n,m}V\\
\Lambda_{n,m} @>\widehat L_{n,m}^{(r)}>> \Lambda_{n,m}
\end{CD}
\end{equation}
is commutative for all $r\in\mathbb{N}$, and consequently that \eqref{hatLnmConj} holds true.
\end{proof}

\begin{remark}
Note that the integral representation \eqref{intRep} is independent of the specific choice of parameter values. When altering $\xi,\eta$ such that a hyperplane $\xi_i=\eta_j$ is crossed, it would seem that we pick up a residue term and thus alter the representation. However, the quasi-invariance conditions \eqref{qInv} ensure that any such residue vanishes.
\end{remark}

\section{Orthogonality relations}
\label{Sec:Orth}
Having proved in Prop.~\ref{Prop:ReprKer} that $SF^{(\theta)}_{n,m}$ is the reproducing kernel of $(\cdot,\cdot)_{n,m}$, the desired orthogonality relations for the super-Jack polynomials are easily inferred from the definition in \eqref{SF} of $SF^{(\theta)}_{n,m}$.

More specifically, setting $p(x,y)=SC^{(\theta)}_\mu(x,y)$ in \eqref{reprEq}, substituting the latter series expansion in \eqref{SF} and comparing coefficients, we obtain
\begin{equation}
\label{SCNorm}
\big(SC^{(\theta)}_\mu,SC^{(\theta)}_\lambda\big)_{n,m} = \delta_{\mu\lambda}|\lambda|!SC_\lambda^{(\theta)}(1^{n+m})\ \ \ (\mu,\lambda\in H_{n,m}),
\end{equation}
where $\delta_{\lambda\mu}$ denotes the Kronecker delta. Keeping the definitions of the two homomorphisms $\epsilon_X$ \eqref{epsX} and $\varphi_{n,m}$ \eqref{varphinm} in mind, it becomes clear from \eqref{SP} that
$$
SP_\lambda^{(\theta)}(1^{n+m})=\epsilon_{n-m/\theta}\big(P_\lambda^{(\theta)}\big),
$$
where by $\epsilon_{n-m/\theta}$ we mean the homomorphism $\Lambda\to\mathbb{C}$ given by $\epsilon_X$ followed by evaluation at $X=n-m/\theta$. Substituting \eqref{SC} in \eqref{SCNorm}, we can now use \eqref{qNorms}--\eqref{spec} to rewrite the right-hand side of the resulting equation in terms of the generalised Pochhammer symbol
\begin{equation}
\label{Pochh}
(a)_\lambda^{(\theta)} = \prod_{i=1}^{\ell(\lambda)}(a-\theta(i-1))_{\lambda_i},
\end{equation}
with $(a)_m$ the ordinary Pochhammer symbol, and the inverse $b_{\lambda}^{(\theta)}$ of the quadratic norm of $P_\lambda^{(\theta)}$, thus arriving at the following result.

\begin{theorem}
\label{Thm:Orth}
As long as $\theta$ is not a negative rational number or zero, we have
$$
\big(SP^{(\theta)}_\mu,SP^{(\theta)}_\lambda\big)_{n,m} = \delta_{\mu\lambda}\frac{(\theta n-m)^{(\theta)}_\lambda}{b_\lambda^{(\theta)}}\ \ (\mu,\lambda\in H_{n,m}).
$$
\end{theorem}

We conclude this section by sketching an alternative approach to this orthogonality result, starting from the bilinear form on $\Lambda$ given by
$$
(p,q)_{p_0} = \epsilon_0(L_pq),
$$
with the homomorphisms $p\mapsto L_p$ and $\epsilon_0:\Lambda\to\mathbb{C}$ characterised by $p_r\mapsto L^{(r)}$ and $\epsilon(p_r)=0$, respectively, for $r\in\mathbb{N}$, and where we think of $p_0$ as a complex parameter. Proceeding as above, it is readily seen that $F^{(\theta,p_0)}(x,y):=\sum_{d=0}^\infty F^{(\theta,p_0)}_d(x,y)$ is the reproducing kernel of $(\cdot,\cdot)_{p_0}$, which, in turn, implies
$$
(P_\mu,P_\lambda)_{p_0} = \delta_{\mu\lambda}\frac{(\theta p_0)^{(\theta)}_\lambda}{b_\lambda^{(\theta)}}.
$$

Setting $p_0=n-m/\theta$, we note that $(\theta p_0)^{(\theta)}_\lambda=(\theta n-m)^{(\theta)}_\lambda=0$ if and only if $(n+1,m+1)\in\lambda$ or equivalently $\lambda\notin H_{n,m}$. (To be precise, the `only if' part of this claim holds true as long as we avoid the $\theta$-values \eqref{thetaVals}.) In other words, the kernel of $(\cdot,\cdot)_{n-m/\theta}$ equals
$$
K_{n,m} := \mathrm{span}\Big\{P^{(\theta)}_\lambda\mid \lambda\notin H_{n,m}\Big\}.
$$
From \cite{SV05} (see Thm.~2), we recall that $K_{n,m}$ is also the kernel of $\varphi_{n,m}:\Lambda\to\Lambda_{n,m}$, so that $(\cdot,\cdot)_{n-m/\theta}$ descends to a non-degenerate bilinear form on the factor space $\Lambda/K_{n,m}\cong\Lambda_{n,m}$ that amounts to $(\cdot,\cdot)_{n,m}$.

\section{Lassalle--Nekrasov correspondence}
\label{Sec:LN}
In this section, we provide a new proof of the Lassalle--Nekrasov correspondence between the deformed trigonometric and rational harmonic Calogero--Moser--Sutherland systems; and, in addition, we show that the correspondence is isometric in a natural sense.

Using the notation \eqref{xExt}--\eqref{parity}, we recall from \cite{SV04,SV05} that if \eqref{pari} is modified such that $\partial^{(1)}_i=(-\theta)^{p(i)}x_i\partial/\partial x_i$ and
\begin{equation*}
\partial^{(r)}_i=\partial^{(1)}_i\partial^{(r-1)}_i-\frac{1}{2}\sum_{j\neq i}(-\theta)^{1-p(j)}\frac{x_i+x_j}{x_i-x_j}\big(\partial^{(r-1)}_i-\partial^{(r-1)}_j\big)
\end{equation*}
for $r>1$, the differential operators
\begin{equation}
\label{cLnmr}
\mathcal{L}_{n,m}^{(r)} = \sum_{i=1}^{n+m} (-\theta)^{-p(i)}\partial_i^{(r)}\ \ (r\in\mathbb{N})
\end{equation}
where $\mathcal{L}_{n,m}^{(2)}=\mathcal{L}_{n,m}$, pairwise commute and are simultaneously diagonalised by the super-Jack polynomials. As we shall see below, the Lassalle--Nekrasov correspondence implies that corresponding quantum integrals for the rational harmonic system are given by
\begin{multline}
\label{cLnmrH}
\mathscr{L}_{n,m}^{(r)} = \mathcal{L}_{n,m}^{(r)}+\frac{1}{2}\big[\mathcal{L}_{n,m}^{(r)},L_{n,m}\big]+\frac{1}{2^22!}\big[\big[\mathcal{L}_{n,m}^{(r)},L_{n,m}\big],L_{n,m}\big]\\
+\cdots+\frac{1}{2^rr!}\big[\cdots\big[\mathcal{L}_{n,m}^{(r)},L_{n,m}\big],\ldots,L_{n,m}\big]\ \ (r\in\mathbb{N}),
\end{multline}
where
$$
\mathscr{L}_{n,m}^{(1)} = \sum_{i=1}^{n+m}x_i\frac{\partial}{\partial x_i}-L_{n,m}.
$$

First, we deduce an alternative description of the map $e^{-L_{n,m}/2}:\Lambda_{n,m}\to\Lambda_{n,m}$. From Prop.~\ref{Prop:SFnm}, it is clear that
\begin{equation*}
\begin{split}
G_{n,m}^{(\theta)}(x,y;z,w) &:= SF_{n,m}^{(\theta)}(x,y;z,w)e^{-p_{2,\theta}(z,w)/2}\\
&= e^{-L_{n,m}(x,y)/2}SF_{n,m}^{(\theta)}(x,y;z,w),
\end{split}
\end{equation*}
and so \eqref{SH} and \eqref{SF} entail the generating function expansion
$$
G_{n,m}^{(\theta)}(x,y;z,w) = \sum_{\lambda\in H_{n,m}}\frac{b_\lambda(\theta)}{(\theta n-m)_\lambda^{(\theta)}}SH_{\lambda}^{(\theta)}(x,y)SP_{\lambda}^{(\theta)}(z,w),
$$
cf.~\eqref{qNorms}--\eqref{spec}, \eqref{SC} and \eqref{Pochh}. Hence, invoking Thm.~\ref{Thm:Orth}, we obtain the following proposition.

\begin{proposition}
\label{Prop:LNRep}
Assuming $\theta$ is not of the form \eqref{thetaVals}, we have
$$
e^{-L_{n,m}(x,y)/2}p(x,y) = \big(G_{n,m}^{(\theta)}(x,y;z,w),p(z,w)\big)_{n,m}
$$
for each $p\in\Lambda_{n,m}$.
\end{proposition}

For suitable $\theta$-values, we proceed to introduce an additional bilinear form on $\Lambda_{n,m}$, which can be viewed as a natural generalisation to the deformed case of the $L^2$ inner product over $\mathbb{R}^N$ with weight function $e^{-x^2/2}\cdot \prod_{1\leq i<j\leq N}|x_i-x_j|^{2\theta}$.

\begin{definition}
\label{Def:AltForm}
Assuming that $\theta\in\mathbb{C}$ is not a negative rational number, that it satisfies $\mathrm{Re}\,\theta<0$ and $\xi\in\mathbb{R}^n$, $\eta\in\mathbb{R}^m$ satisfy \eqref{xieta}, we define a bilinear form on $\Lambda_{n,m}$ by
$$
\{p,q\}_{n,m} = M_{n,m}^{-1} \int_{\mathbb{R}^n+i\xi}\int_{\mathbb{R}^m+i\eta} p(x,y)q(x,y)\frac{e^{-x^2/2+\theta^{-1}y^2/2}}{A_{n,m}(x,y)}dxdy\ \ (p,q\in\Lambda_{n,m}),
$$
where the value of $M_{n,m}$ is given by \eqref{MnmExpr}--\eqref{Cnm}.
\end{definition}

We are now ready to state and prove the main result of this section.

\begin{theorem}
\label{Thm:LN}
For $\theta$ not of the form \eqref{thetaVals}, the map $e^{-L_{n,m}/2}:\Lambda_{n,m}\to\Lambda_{n,m}$ intertwines the deformed trigonometric and rational harmonic Calogero--Moser--Sutherland operators given by \eqref{cLnmr} and \eqref{cLnmrH}, respectively. More precisely, the diagram
\begin{equation*}
\begin{CD}
\Lambda_{n,m} @>e^{-L_{n,m}/2}>> \Lambda_{n,m}\\
@V\mathcal{L}_{n,m}^{(r)}VV @VV\mathscr{L}_{n,m}^{(r)}V\\
\Lambda_{n,m} @>e^{-L_{n,m}/2}>> \Lambda_{n,m}
\end{CD}
\end{equation*}
is commutative for all $r\in\mathbb{N}$.

If, in addition, $\mathrm{Re}\,\theta<0$, we have
\begin{equation*}
\big\{e^{-L_{n,m}/2}p,e^{-L_{n,m}/2}q\big\}_{n,m} = (p,q)_{n,m}\ \ (p,q\in\Lambda_{n,m}).
\end{equation*}
\end{theorem}

\begin{proof}
Using the formula $\mathrm{Ad}_{e^X} = e^{\mathrm{ad}_X}$, with $X$ the operator of multiplication by $p_{2,\theta}/2$, as well as the
fact that $\mathrm{ad}_X$ lowers the order of a differential operator by at least one, we deduce
\begin{multline*}
e^{p_{2,\theta}/2}\mathcal{L}_{n,m}^{(r)}e^{-p_{2,\theta}/2}
= \mathcal{L}_{n,m}^{(r)}+\frac{1}{2}\big[p_{2,\theta},\mathcal{L}_{n,m}^{(r)}\big]+\frac{1}{2^22!}\big[p_{2,\theta},\big[p_{2,\theta},\mathcal{L}_{n,m}^{(r)}\big]\big] \\
  + \cdots 
+ \frac{1}{2^rr!}\big[p_{2,\theta},\ldots,\big[p_{2,\theta},\mathcal{L}_{n,m}^{(r)}\big]\cdots\big].
\end{multline*}
As a direct consequence of the definition in \eqref{SF} of $SF_{n,m}^{(\theta)}$ and the fact that the super-Jack polynomials are joint eigenfunctions of the operators $\mathcal{L}_{n,m}^{(r)}$, we get
$$
\mathcal{L}_{n,m}^{(r)}(x,y)SF_{n,m}^{(\theta)}(x,y;z,w) = \mathcal{L}_{n,m}^{(r)}(z,w)SF_{n,m}^{(\theta)}(x,y;z,w).
$$
Combining the previous two formulae with Prop.~\ref{Prop:SFnm}, we infer
\begin{multline*}
\mathcal{L}_{n,m}^{(r)}(z,w) G_{n,m}^{(\theta)}(x,y;z,w)\\
= e^{-p_{2,\theta}(z,w)/2}\Big(e^{p_{2,\theta}/2}\mathcal{L}_{n,m}^{(r)}e^{-p_{2,\theta}/2}\Big)(z,w)SF_{n,m}^{(\theta)}(x,y;z,w)\\
= \mathscr{L}_{n,m}^{(r)}(x,y) G_{n,m}^{(\theta)}(x,y;z,w).
\end{multline*}
Since Thm.~\ref{Thm:Orth} and the pertinent joint eigenfunction property imply
$$
\big(\mathcal{L}_{n,m}^{(r)}p,q\big)_{n,m} = \big(p,\mathcal{L}_{n,m}^{(r)}q\big)_{n,m}\ \ (p,q\in\Lambda_{n,m}),
$$
it follows from Prop.~\ref{Prop:LNRep} and our reasoning thus far that
\begin{equation*}
\begin{split}
\big(\mathscr{L}_{n,m}^{(r)} e^{-L_{n,m}/2}\big)(p) &= \big(\mathscr{L}_{n,m}^{(r)}(x,y) G_{n,m}^{(\theta)}(x,y;z,w),p(z,w)\big)_{n,m}\\
&= \big(G_{n,m}^{(\theta)}(x,y;z,w),\mathcal{L}_{n,m}^{(r)}(z,w)p(z,w)\big)_{n,m}\\
&= (e^{-L_{n,m}/2}\mathcal{L}_{n,m}^{(r)})(p)
\end{split}
\end{equation*}
for all $r\in\mathbb{N}$ and $p\in\Lambda_{n,m}$.

Finally, if $\mathrm{Re}\,\theta<0$, it is clear from Prop.~\ref{Prop:intRep} and Def.~\ref{Def:AltForm} that the map $e^{-L_{n,m}/2}:\Lambda_{n,m}\to\Lambda_{n,m}$ becomes an isometry when the domain is equipped with the bilinear form $(\cdot,\cdot)_{n,m}$ and the codomain with $\{\cdot,\cdot\}_{n,m}$
\end{proof}

\begin{remark}
This is precisely the Lassalle--Nekrasov correspondence we had in mind, and, while the first part of the result should be compared with Thm.~6 in \cite{FHV21}, the above proof runs in parallel with that of Thm.~4 in \cite{FHV21}, which pertains to the ordinary undeformed case.
\end{remark}

\begin{remark}
As detailed in Thm.~8 in \cite{FHV21}, it is readily inferred that the deformed rational harmonic Calogero--Moser--System is integrable. This was first proved independently by Desrosiers and the author \cite{DH12} and Feigin \cite{Fei12}; see also Berest and Chalykh \cite{BC20}.
\end{remark}

\begin{corollary}
For $\theta\in\mathbb{C}$ satisfying $\mathrm{Re}\,\theta<0$ while not being equal to a negative rational number, we have
$$
\big\{SH^{(\theta)}_\mu,SH^{(\theta)}_\lambda\big\}_{n,m} = \delta_{\mu\lambda}\frac{(\theta n-m)^{(\theta)}_\lambda}{b_\lambda^{(\theta)}}\ \ (\mu,\lambda\in H_{n,m}).
$$
\end{corollary}

\begin{proof}
Taking $p=SH^{(\theta)}_\mu$ and $q=SH^{(\theta)}_\lambda$ in Thm.~\ref{Thm:LN} and invoking Thm.~\ref{Thm:Orth}, we immediately obtain the claim.
\end{proof}

\section{Outlook}
\label{Sec:Out}
In this paper, we have worked with deformed Calogero--Moser--Sutherland operators and corresponding eigenfunctions associated with deformations of root systems of type $A$. It seems plausible that our results can be generalised to the $BC$ case, with the super-Jack polynomials being replaced by the super-Jacobi polynomials introduced in \cite{SV09}. Indeed, the constructions and results from \cite{BF98,DH12,SV15} that we have relied on are available also in the $BC$ case.

We also note that Feigin's \cite{Fei12} approach to integrability of deformed Calogero--Moser--Sutherland operators, using special representations of rational Cherednik algebras, might provide a different way to establish the present results and point the way towards further generalisations; and the recent orthogonality results in \cite{AHL21} on super-Macdonald polynomials hint at generalisations to the difference case.

Due, in particular, to parameter-independent singularities of eigenfunctions, deformed Calogero--Moser--Sutherland systems were for a long time seen as problematic to interpret within (quantum) physics. However, in recent years, the deformed trigonometric and even elliptic systems have naturally appeared in a quantum field theory formulation of the ordinary systems \cite{AL17,BLL20} as well as in the context of super-symmetric gauge theories \cite{Nek17,CKL20}. It would be interesting to explore possible connections to the results presented in this paper, not least the Lassalle--Nekrasov correspondence between deformed trigonometric and rational harmonic systems.

It is also interesting to note that, in contrast to the situation in \cite{AHL19}, the bilinear forms introduced in Defs.~\ref{Def:Form} and \ref{Def:AltForm} are nondegenerate for generic $\theta$-values and it is not obvious whether a natural positive definite inner product can be extracted, e.g.~by restricting attention to particular parameter values and a subspace of $\Lambda_{n,m}$. Insights into this problem might provides clues on potential (physical) interpretations of the results obtained in this paper and vice versa.

\section*{Acknowledgemnents}
I would like to thank two anonymous referees for a number of helpful comments.

\begin{appendix}

\section{On convergence of generalised hypergeometric series}
\label{App:Conv}
In this appendix, we study convergence properties of the generalised hypergeometric series $SF^{(\theta)}_{n,m}(x,y;z,w)$ \eqref{SF}, restricting, for simplicity, attention to $\theta>0$. Specifically, we prove the proposition below by following the approach of Desrosiers and Liu \cite{DL15} (see Appendix B), who considered generalised hypergeometric series ${}_pSF_q^{(\alpha)}(a_1,\ldots,a_p;b_1,\ldots,b_q;x,y)$, depending on $p+q$ parameters in addition to $\alpha=1/\theta$, where the special case $p=q=0$ corresponds to the specialisation $(z,w)=(1^n,1^m)$ of $SF^{(\theta)}_{n,m}(x,y;z,w)$.

\begin{proposition}
Assume that $\theta>0$ is such that $\theta\neq i/j$ for any $1\leq i\leq m$, $1\leq j\leq n$. Then $SF^{(\theta)}_{n,m}(x,y;z,w)$ is analytic for all $(x,y),(z,w)\in\mathbb{C}^n\times\mathbb{C}^m$.
\end{proposition}

\begin{proof}
Letting $\chi^\lambda_\mu(\theta)$ denote the coefficient of $p_\mu$ in the power sum expansion of the Jack symmetric function $P_\lambda^{(\theta)}$, we get from \eqref{varphinm} and \eqref{SP} that
$$
SP_\lambda^{(\theta)}(x,y) = \sum_{\mu}\chi^\lambda_\mu(\theta)p_{\mu,\theta}(x,y),
$$
where the sum extends over partitions $\mu$ such that $|\mu|=|\lambda|$. Just as in Lemma B.1 in \cite{DL15}, we use the Cauchy-Schwartz inequality to deduce the bound
$$
\big|SP_\lambda^{(\theta)}(x,y)\big|^2\leq \left(\sum_\mu \chi^\lambda_\mu(\theta)^2\theta^{-\ell(\mu)}z_\mu\right)\cdot \left(\sum_\mu |p_{\mu,\theta}(x,y)|^2\theta^{\ell(\mu)}z_\mu^{-1}\right)
$$
and, since the former sum amounts to the norm of $P_\lambda$ with respect to the scalar product $\langle\cdot,\cdot\rangle$ (cf.~\eqref{pOrth}), Stanley's formula \eqref{qNorms} amounts to
$$
\sum_\mu \chi^\lambda_\mu(\theta)^2\theta^{-\ell(\mu)}z_\mu = \frac{1}{b_\lambda^{(\theta)}}.
$$
Moreover, we have
\begin{equation*}
\begin{split}
|p_{\mu,\theta}(x,y)|^2\theta^{\ell(\mu)} &\leq p_{\mu,-\theta}((|x_1|,\ldots,|x_n|),(|y_1|,\ldots,|y_m|))^2\theta^{\ell(\mu)}\\
&\leq ||(x,y)||_\infty^{2|\mu|}\cdot \big(\sqrt{\theta}n+m/\sqrt{\theta}\big)^{2\ell(\mu)}.
\end{split}
\end{equation*}
Using $\ell(\mu)\leq |\mu|=|\lambda|$ and $\sum_{|\mu|=|\lambda|}z_\mu^{-1}=1$, we thus arrive at the bound
$$
|SP_\lambda^{(\theta)}(x,y)| \leq \frac{1}{\sqrt{b_\lambda^{(\theta)}}} \big(||(x,y)||_\infty\cdot \big(\sqrt{\theta}n+m/\sqrt{\theta}\big)\big)^{|\lambda|},
$$
which, when combined with \eqref{spec} with $X=n-m/\theta$ and \eqref{SC} yields
\begin{equation*}
\begin{split}
S^{(\theta)}(x,y;z,w) &:= \sum_{\lambda\in H_{n,m}}\left|\frac{1}{|\lambda|!}\frac{SC^{(\theta)}_\lambda(x,y)SC^{(\theta)}_\lambda(z,w)}{SC^{(\theta)}_\lambda(1^{n+m})}\right|\\
&\leq \sum_{\lambda\in H_{n,m}}\frac{\Big(||(x,y)||_\infty\cdot ||(z,w)||_\infty\cdot\big(\sqrt{\theta}n+m/\sqrt{\theta}\big)^2\Big)^{|\lambda|}}{\prod_{s\in\lambda}|\theta(n-l^\prime(s))-(m-a^\prime(s))|}.
\end{split}
\end{equation*}

For a partition $\lambda\in H_{n,m}$, we let
$$
e(\lambda) = (\langle\lambda_1-m\rangle,\cdots,\langle\lambda_n-m\rangle),\ \ \ s(\lambda) = (\langle\lambda^\prime_1-n\rangle,\cdots,\langle\lambda^\prime_m-n\rangle),
$$
with $\langle a\rangle=\max(0,a)$, so that $e(\lambda)$ and $s(\lambda)$ correspond to the set of boxes located to the `east' and `south', respectively, of the $m\times n$ rectangle $(m^n)$ in the diagram of $\lambda$, cf.~Fig.~1 in \cite{AHL19}. Then, we have
\begin{multline*}
\prod_{s\in\lambda}|\theta(n-l^\prime(s))-(m-a^\prime(s))|\\
= (\theta n)^{(\theta)}_{e(\lambda)}\cdot (\theta)^{|s(\lambda)|}(m/\theta)^{(1/\theta)}_{s(\lambda)}\cdot \prod_{s\in \lambda\cap(m^n)}|\theta(n-l^\prime(s))-(m-a^\prime(s))|.
\end{multline*}
As long as the restrictions \eqref{thetaVals} are in place, each factor in the product over boxes in $\lambda\cap(m^n)$ in the right-hand side is non-zero and independent of $\lambda$, which implies that the product is uniformly bounded below by some positive constant. Furthermore, under our assumption $\theta>0$, it is clear from \eqref{Pochh} that
$$
(\theta n)^{(\theta)}_{e(\lambda)} \geq \min\{1,\theta\}^{|e(\lambda)|}\cdot e(\lambda)!,\ \ (\theta)^{|s(\lambda)|}(m/\theta)^{(1/\theta)}_{s(\lambda)} \geq \min\{1,\theta\}^{|s(\lambda)|}\cdot s(\lambda)!.
$$
Hence, we can find constants $C,r>0$ such that
\begin{equation*}
\begin{split}
S^{(\theta)}(x,y;z,w) &\leq C\big(1+\big(||(x,y)||_\infty\cdot ||(z,w)||_\infty\big)^{mn}\big)\\
&\quad \cdot\sum_{(\mu,\nu)\in\mathbb{Z}_{\geq 0}^n\times\mathbb{Z}_{\geq 0}^m} \frac{\big(r\cdot ||(x,y)||_\infty\cdot ||(z,w)||_\infty\big)^{|\mu|+|\nu|}}{\mu!\nu!}.
\end{split}
\end{equation*}
Since the sum converges locally uniformly to $\exp((n+m)r\cdot ||(x,y)||_\infty\cdot ||(z,w)||_\infty)$, the assertion follows.
\end{proof}

\end{appendix}

\bibliographystyle{amsalpha}

\begin{thebibliography}{JHLM06}

\bibitem[AHL19]{AHL19}
Atai, F., Halln\"as, M., Langmann, E.:
Orthogonality of super-Jack polynomials and a Hilbert space interpretation of deformed Calogero--Moser--Sutherland operators. 
Bull.~Lond.~Math.~Soc.~51, 353--370 (2019).

\bibitem[AHL21]{AHL21}
Atai, F., Halln\"as, M., Langmann, E.:
Super-Macdonald polynomials: Orthogonality and Hilbert space interpretation.
Commun.~Math.~Phys.~388, 435--468 (2021).

\bibitem[AL17]{AL17}
Atai, F., Langmann, E.:
Deformed Calogero-Sutherland model and fractional quantum Hall effect.
J.~Math.~Phys.~58, 011902, 27pp.~(2017).

\bibitem[BF97]{BF97}
Baker, T.H., Forrester, P.J.:
The Calogero-Sutherland model and generalized classical polynomials.
Commun.~Math.~Phys.~188, 175--216 (1997).

\bibitem[BF98]{BF98}
Baker, T.~H., Forrester, P.~J.:
Nonsymmetric Jack polynomials and integral kernels. 
Duke Math.~J.~95, 1--50 (1998).

\bibitem[BC20]{BC20}
Berest, Y., Chalykh, O.: Deformed Calogero--Moser operators and ideals of rational Cherednik algebras.
arXiv:2002.08691 (2020).

\bibitem[BLL20]{BLL20}
Berntson, B.K., Langmann, E., Lenells, J.:
Nonchiral intermediate long-wave equation and inter-edge effects in narrow quantum Hall systems.
Phys.~Rev.~B 102, 155308, 14pp.~(2020).

\bibitem[BR23]{BR23}
Brennecken, D, R\"osler, M.:
The Dunkl--Laplace transform and Macdonald's hypergeometric series.
Trans.~Amer.~Math.~Soc.~376, 2419--2447 (2023).

\bibitem[CKL20]{CKL20}
Chen, H.-Y., Kimura, T., Lee, N.:
Quantum elliptic Calogero-Moser systems from gauge origami.
J.~High Energy Phys.~2020, 108, 40 pp.~(2020).

\bibitem[CFV98]{CFV98}
Chalykh, O., Feigin, M., Veselov, A.P.:
New integrable generalizations of Calogero-Moser quantum problem.
J.~Math.~Phys.~39, 695--703 (1998).

\bibitem[DH12]{DH12}
Desrosiers, P., Halln\"as, M.:
Hermite and Laguerre symmetric functions associated with operators of Calogero--Moser--Sutherland type.
SIGMA Symmetry Integrability Geom. Methods Appl.~8, Paper 049, 51 pp.~(2012). 

\bibitem[DL15]{DL15}
Desrosiers, P., Liu, D.-Z.:
Selberg integrals, super-hypergeometric functions and applications to $\beta$-ensembles of random matrices.
Random Matrices Theory Appl.~4, 1550007, 59 pp.~(2015).

\bibitem[vDie97]{vDie97}
van Diejen, J.F.:
Confluent hypergeometric orthogonal polynomials related to the rational quantum Calogero system with harmonic confinement.
Commun.~Math.~Phys.~188, 467--497 (1997).

\bibitem[Dun89]{Dun89}
Dunkl, C.F.:
Differential-difference operators associated to reflection groups. 
Trans.~Amer.~Math.~Soc.~311, 167--183 (1989).

\bibitem[Dun91]{Dun91}
Dunkl, C.F.:
Integral kernels with reflection group invariance. 
Canad.~J.~Math.~43, 1213--1227 (1991). 

\bibitem[Eti07]{Eti07}
Etingof, P.:
Calogero--Moser systems and representation theory. 
Zurich Lectures in Advanced Mathematics.
European Mathematical Society (EMS), Zürich (2007).

\bibitem[Fei12]{Fei12}
Feigin, M.: Generalized Calogero--Moser systems from rational Cherednik algebras. 
Selecta Math.~(N.S.)~18, 253--281 (2012).

\bibitem[FHV13]{FHV13}
Feigin, M., Halln\"as, M., Veselov, A.P.:
Baker--Akhiezer functions and generalised Macdonald--Mehta integrals.
J.~Math.~Phys.~54, 052106 (2013).

\bibitem[FHV21]{FHV21}
Feigin, M., Halln\"as, M., Veselov, A.P.:
Quasi-invariant Hermite polynomials and Lassalle-Nekrasov correspondence. 
Commun.~Math.~Phys.~386, 107--141 (2021).

\bibitem[FV02]{FV02}
Feigin, M., Veselov, A. P.:
Quasi-invariants of Coxeter groups and m-harmonic polynomials. 
Int.~Math.~Res.~Not.~2002, 521--545 (2002).

\bibitem[FV03]{FV03}
Feigin, M., Veselov, A. P.:
Quasi-invariants and quantum integrals of the deformed Calogero--Moser systems. 
Int.~Math.~Res.~Not.~2003, 2487--2511 (2003).

\bibitem[Hec91]{Hec91}
Heckman, G.~J.:
A remark on the Dunkl differential-difference operators, in: Harmonic analysis on reductive groups. 
Progr.~Math.~101, 181--191, Birkhäuser Boston, Boston, MA (1991). 

\bibitem[Jac70]{Jac70}
Jack, H.:
A class of symmetric polynomials with a parameter.
Proc.~Roy.~Soc.~Edinburgh Sect.~A 69, 1--18 (1970/1971).

\bibitem[JHLM06]{JHLM06}
Jack, Hall-Littlewood and Macdonald polynomials. 
Proceedings of the workshop held in Edinburgh, September 23–26, 2003. Edited by V.B. Kuznetsov and S.~Sahi. Contemp.~Math.~417.
American Mathematical Society, Providence, RI (2006).

\bibitem[Kan93]{Kan93}
Kaneko, J.:
Selberg integrals and hypergeometric functions associated with Jack polynomials.
SIAM J.~Math.~Anal.~24, 1086--1110 (1993).

\bibitem[KOO98]{KOO98}
Kerov, A., Okounkov, A., Olshanski, G.:
The boundary of the Young graph with Jack edge multiplicities.
Internat.~Math.~Res.~Notices 1998, 173--199 (1998).

\bibitem[Las91]{Las91}
Lassalle, M.:
Polyn\^omes de Hermite g\'en\'eralis\'es.
C.~R.~Acad.~Sci.~Paris S\'er.~I Math.~313, 579--582 (1991).

\bibitem[Mac80]{Mac80}
Macdonald, I.G.:
The volume of a compact Lie group. 
Invent.~Math.~56, 93--95 (1980).

\bibitem[Mac95]{Mac95}
Macdonald, I.G.:
Symmetric functions and Hall polynomials, 2nd edn.
Oxford University Press, New York (1995).

\bibitem[Mac13]{Mac13}
Macdonald, I.G.:
Hypergeometric functions I.
arXiv:1309.4568 (2013).

\bibitem[Nek97]{Nek97}
Nekrasov, N.:
On a duality in Calogero--Moser--Sutherland systems.
arXiv:hep-th/9707111 (1997).

\bibitem[Nek17]{Nek17}
Nekrasov, N.:
BPS/CFT correspondence V: BPZ and KZ equations from qq-characters,
arXiv:1711.11582 (2017).

\bibitem[OO97]{OO97}
Okounkov, A., Olshanski, G.:
Shifted Jack polynomials, binomial formula, and applications. 
Math.~Res.~Lett.~4, 69--78 (1997).

\bibitem[OP83]{OP83}
Olshanetsky, M.~A., Perelomov, A.~M.:
Quantum integrable systems related to Lie algebras. 
Phys.~Rep.~94, 313--404 (1983).

\bibitem[Opd93]{Opd93}
Opdam, E.~M.:
Dunkl operators, Bessel functions and the discriminant of a finite Coxeter group. 
Compositio Math.~85, 333--373 (1993).

\bibitem[Ros98]{Ros98}
R\"osler, M.:
Generalized Hermite polynomials and the heat equation for Dunkl operators. 
Commun.~Math.~Phys.~192, 519--542 (1998). 

\bibitem[Rui99]{Rui99}
Ruijsenaars, S.N.M.:
Systems of Calogero--Moser type, in: Particles and fields.
CRM Ser.~Math.~Phys., 251--352, Springer, New York (1999).

\bibitem[Ser01]{Ser01}
Sergeev, A.N.:
Superanalogs of the Calogero operators and Jack polynomials.
J.~Nonlinear Math.~Phys.~8, 59--64 (2001).

\bibitem[Ser02]{Ser02}
Sergeev, A.N.:
The Calogero operator and Lie superalgebras.
Theoret.~and Math.~Phys.~131, 747--764 (2002).

\bibitem[SV04]{SV04}
Sergeev, A.N., Veselov, A.P.:
Deformed quantum Calogero-Moser problems and Lie superalgebras.
Commun.~Math.~Phys.~245, 249--278 (2004).

\bibitem[SV05]{SV05} 
Sergeev, A.N., Veselov, A.P.:
Generalised discriminants, deformed Calogero-Moser-Sutherland operators and super-Jack polynomials.
Adv.~Math.~192, 341--375 (2005).

\bibitem[SV09]{SV09}
Sergeev, A.N., Veselov, A.P.:
$BC_\infty$ Calogero--Moser operator and super Jacobi polynomials.
Adv.~Math.~222, 1687--1726 (2009). 

\bibitem[SV15]{SV15}
Sergeev, A.N., Veselov, A.P.:
Dunkl operators at infinity and Calogero--Moser systems.
Internat.~Math.~Res.~Notices 2015, 10959--10986 (2015).

\bibitem[Sta89]{Sta89}
Stanley, R.P.:
Some combinatorial properties of Jack symmetric functions. 
Adv.~Math.~77, 76--115 (1989). 

\bibitem[Sta99]{Sta99}
Stanley, R.P.:
Enumerative combinatorics. Vol.~2.
Cambridge University Press, Cambridge (1999).

\end{thebibliography}

\end{document}